\documentclass[a4paper,11pt]{article}

\usepackage{amsmath,amsthm,amssymb}
\usepackage{graphicx,subcaption,tikz}
\usepackage{xspace,paralist}							                              
\usepackage[shortlabels]{enumitem}	
\usepackage{hyperref}       
\hypersetup{
    colorlinks=true,
    citecolor = blue,
    linkcolor=blue,
    filecolor=magenta,
    urlcolor=cyan,
    bookmarks=true,
}
\usepackage{todonotes}

\newtheorem{theorem}{Theorem}

\newtheorem{lemma}{Lemma}
\newtheorem{claim}{Claim}

\newtheorem{conjecture}{Conjecture}

\setlist[1]{itemsep=1em}

\title{{\bf The optimal $\chi$-bound for $(P_7,C_4,C_5)$-free graphs}}

\author{
Shenwei Huang\thanks{College of Computer Science, Nankai University, Tianjin 300350, China. Email: shenweihuang@nankai.edu.cn. Supported by the National Natural Science Foundation of China (12171256).}
}

\date{December 10, 2022}

\begin{document}

\maketitle

\begin{abstract} 
In this paper, we give an optimal $\chi$-binding function for the class of $(P_7,C_4,C_5)$-free graphs.
We show that every $(P_7,C_4,C_5)$-free graph $G$ has $\chi(G)\le \lceil \frac{11}{9}\omega(G) \rceil$.
To prove the result, we use a decomposition theorem obtained in 
[K. Cameron and S. Huang and I. Penev and V. Sivaraman, 
The class of  $({P}_7,{C}_4,{C}_5)$-free graphs: Decomposition, algorithms, and $\chi$-boundedness, Journal of Graph Theory 93, 503--552, 2020]
combined with careful inductive arguments and a nontrivial use of the K\"{o}nig theorem for bipartite matching.
\end{abstract}

\section{Introduction}

All graphs in this paper are finite and simple. For general graph theory notation we follow~\cite{BM08}. A \emph{$q$-coloring} of a graph $G$ assigns a color from $\{1,\ldots,q\}$ to each vertex of $G$ such that adjacent vertices are assigned different colors. We say that a graph $G$ is \emph{$q$-colorable} if $G$ admits a $q$-coloring. 
The \emph{chromatic number} of a graph $G$, denoted by $\chi(G)$, is the minimum number $q$ for which $G$ is $q$-colorable.
The {\em clique number} of $G$, denoted by $\omega(G)$, is the size of a largest clique in $G$.  
A graph $G$ is \emph{perfect} if $\chi(H)=\omega(H)$ for each induced subgraph $H$ of $G$

Let $P_n$, $C_n$ and $K_n$ denote the path, cycle and complete graph on $n$ vertices, respectively.
For a set $\mathcal{H}$ of graphs, we say that $G$ is {\em $\mathcal{H}$}-free if $G$ contains no induced subgraph isomorphic to $H\in \mathcal{H}$. If $\mathcal{H}=\{H_1,\ldots,H_p\}$, we simply write $G$ is $(H_1,\ldots,H_p)$-free instead.
A graph class is {\em hereditary} if it is $\mathcal{H}$-free for some set $\mathcal{H}$ of graphs.
A hereditary graph class $\mathcal{G}$ is \emph{$\chi$-bounded} if there is a function $f$ such that $\chi(G)\leq f(\omega(G))$ for every $G\in\mathcal{G}$.  The function $f$ is called a \emph{$\chi$-binding function} for $\mathcal{G}$. 
If $f$ can be chosen to be a polynomial function, then $\mathcal{G}$ is called {\em polynomially $\chi$-bounded}.
It is easy to see that a necessary condition for the class of $H$-free graphs to be $\chi$-bounded is that $H$ is a forest.
Gy{\'a}rf{\'a}s conjectured that this condition is also sufficient.

\begin{conjecture}[Gy{\'a}rf{\'a}s \cite{Gy73}]
For every forest $T$, the class of $T$-free graphs is $\chi$-bounded.
\end{conjecture}

Gy{\'a}rf{\'a}s \cite{Gy87} proved the conjecture for $T=P_t$: every $P_t$-free graph
$G$ has $\chi(G)\le (t-1)^{\omega(G)-1}$.  
Note that this $\chi$-binding function is exponential in $\omega(G)$. 
Esperet asked whether every $\chi$-bounded hereditary class is polynomially $\chi$-bounded. 
Recently, this question was answered negatively by Bria\'{n}ski, Davies and Walczak \cite{BDW22}.
However, the question is still wide open for $P_t$-free graphs.

The answer is trivial for $t\le 4$ (since $P_4$-free graphs are perfect).
For $t= 5$, there is a recent result by Scott, Seymour and Spirkl \cite{SSS22} that  gives a quasi-polynomial bound. 
Apart from these, not much is known. Therefore, researchers started to investigate subclasses of $P_t$-free graphs.
One interesting line is to consider $P_t$-free graphs without certain induced cycles. 
For instance, any $(P_5,C_3)$-free graph is 3-colorable (the bound is attained by $C_5$) \cite{WS01},
and any $(P_5,C_4)$-free graph has $\chi(G)\le \lceil \frac{5}{4}\omega(G) \rceil$ (the equal-size blowups of $C_5$ shows that
the bound is optimal) \cite{CKS07}. For $t=6$, optimal bounds have been obtained for $C_3$ and $C_4$ as well \cite{BKM06,KMa18}. 
For $t\ge 7$, Gravier, Ho{\`a}ng and Maffray \cite{GHM03} showed that every $(P_t,C_3)$-free graph $G$ has $\chi(G)\le t-2$.
However, it is not known whether this is optimal. Recently, Kathie et al. \cite{CHPS20} showed that any $(P_7,C_4,C_5)$-free $G$ has $\chi(G)\le \frac{3}{2}\omega(G)$.

\subsection*{Our Contribution} In this paper, we give an optimal $\chi$-bound for the class of $(P_7,C_4,C_5)$-free graphs.
In particular, we prove the following theorem.

\begin{theorem}\label{thm:chibound}
Let $G$ be a $(P_7,C_4,C_5)$-free graph. Then $\chi(G)\le \lceil \frac{11}{9}\omega(G) \rceil$.
\end{theorem}

To prove the theorem, we use a decomposition theorem obtained in \cite{CHPS20} combined with careful inductive arguments and a nontrivial use of the K\"{o}nig theorem for bipartite matching.  It is shown in \cite{CHPS20} that for every positive integer $k$
there is a $(P_7,C_4,C_5)$-free graph $G_k$ such that $\omega(G_k)=3k$ and $\chi(G_k)\ge \frac{11}{9} \omega(G_k)$. 
Therefore, our result gives an optimal bound.

\section{Preliminaries}\label{sec:pre}

Let $G=(V,E)$ be a graph. 
The \emph{neighborhood} of a vertex $v$ is denoted by $N(v)=\{u\; |\; uv\in E\}$ and its degree by $d(v)=|N(v)|$.
The {\em closed neighborhood} of $v$, denoted by $N[v]$, is $N(v)\cup \{v\}$.
For a set $X\subseteq V$, we write $N(X)=\bigcup_{v\in X}N(v)\setminus X$.
For $x\in V$ and $S\subseteq V$, we let $N_S(x)$ be the set of neighbors of $x$ that are in $S$, that is, $N_S(x)=N_G(x)\cap S$.
Define $d_S(v)=|N_S(v)|$.
A vertex~$u$ is \emph{universal} in $G$ if $d_G(u)=|G|-1$.

For $S\subseteq V$, the subgraph \emph{induced} by $S$, is denoted by $G[S]$.
The {\it complement} of $G$ is the graph $\overline{G}$ with vertex set $V$ and edge set $\{uv\; |\; uv\notin E\}$.
A clique $K\subseteq  V$ is a \emph{clique cutset} if $G-K$ has more connected components than $G$. 
A clique $K$ is {\em maximal} if for any $v\notin K$, $K\cup \{v\}$ is not a clique.
A subset is {\em stable} if no two vertices in the set are adjacent. A stable set $S$ is {\em strong}
if it intersects every maximum clique of $G$, i.e., $\omega(G-S)=\omega(G)-1$.

For $X,Y\subseteq V$, we say that $X$ is \emph{complete} (resp. \emph{anticomplete}) to $Y$
if every vertex in $X$ is adjacent (resp. non-adjacent) to every vertex in $Y$.
Given a graph $G=(\{v_1,\ldots,v_t\},E(G))$ of order $t$ and graphs $H_1,\ldots,H_t$, we say that {\em a substitution of $G$ using
$H_1,\ldots,H_t$} is the graph obtained from $G$ by replacing $v_i$ with a copy of $H_i$ such that
$V(H_i)$ and $V(H_j)$ are complete if $v_iv_j\in E(G)$, and anticomplete if $v_iv_j\notin E(G)$.
A \emph{blowup} of $G$ is a substitution of $G$ using complete graphs $K_{s_1},\ldots, K_{s_t}$ of size $s_1,\ldots,s_t\ge 0$, respectively. Note that in the definition of blowup we allow complete graphs of size 0, which is equivalent to the operation of removing the corresponding vertex. Therefore, any proper induced subgraph of $G$ is a blowup of $G$ under our definition.
If $s_i\ge 1$ for each $1\le i\le t$, we say the resulting graph is a {\em nonempty blowup} of $G$.
If $s_1=\cdots =s_t=x$, we denote the resulting graph by $G[K_x]$ and call $G[K_x]$ an equal-size blowup of $G$.

\section{The Main Result}

In this section, we prove \autoref{thm:chibound}.
For that purpose, we need a decomposition theorem for $(P_7,C_4,C_5)$-free graphs.
To state the decomposition theorem, we need to introduce two special graphs called the emerald and 7-bracelets.
The {\em emerald} is a 11-vertex 4-regular graph shown in \autoref{fig:emerald}.
We denote this graph by $E$.

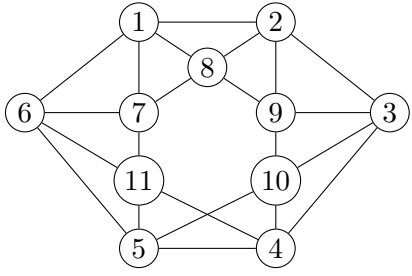
\begin{figure}[t]
	 \centering
	 \begin{tikzpicture}[scale=0.6]
	  \tikzstyle{vertex}=[draw, circle, fill=white!100, minimum width=4pt,inner sep=2pt]
	  
	  \node[vertex] (v1) at (-1.5,2) {1};
	  \node[vertex] (v2) at (1.5,2) {2};
	  \node[vertex] (v3) at (4,0) {3};
	  \node[vertex] (v4) at (1.5,-3) {4};
	  \node[vertex] (v5) at (-1.5,-3) {5};
	  \node[vertex] (v6) at (-4,0) {6};
	  \draw (v1)--(v2)--(v3)--(v4)--(v5)--(v6)--(v1);
	  
	  \node[vertex] (v14) at (-1.5,0) {7};
	  \draw (v14)--(v1) (v14)--(v6);
	  \node[vertex] (v25) at (1.5,0) {9};
	  \draw (v25)--(v2) (v25)--(v3);
	  \node[vertex] (v36) at (0,1) {8};
	  \draw (v36)--(v1) (v36)--(v2) (v36)--(v14) (v36)--(v25);
	  \node[vertex] (s1) at (-1.5,-1.5) {11};
	  \draw (s1)--(v4) (s1)--(v5) (s1)--(v6) (s1)--(v14);
	  \node[vertex] (s2) at (1.5,-1.5) {10};
	  \draw (s2)--(v3) (s2)--(v4) (s2)--(v5) (s2)--(v25);
	 \end{tikzpicture}
	 \caption{The emerald $E$.}
	 \label{fig:emerald}
	\end{figure}
	
\begin{figure}[t]
	\center
	\begin{tikzpicture}[scale=0.8]
	\tikzstyle{vertex}=[draw, circle, fill=black!15, minimum size=80pt, inner sep=0pt]
	\tikzstyle{set}=[draw,rectangle,minimum size=12pt]

	\node [vertex,label=above:$A_5$] (A5) at (3,12) {};
	
	\node [vertex,label=above:$A_4$] (A4) at (7,12) {};
	
	\node [vertex,label=right:$A_3$] (A3) at (10,8) {};
	
	\node [set,label=below:$A_3^-$] (A3-) at (9.5,8) {};
	
	\node [set,label=below:$A_3^0$] (A30) at (10.5,8) {};
	
 	\node [vertex,label=left:$A_6$] (A6) at (0,8) {};
	
	\node [set,label=below:$A_6^+$] (A6+) at (0.5,8) {};
	
	\node [set,label=below:$A_6^0$] (A60) at (-0.5,8) {};
	
	\node [vertex,label=right:$A_2$] (A2) at (10,3) {};
	
	\node [set,label=below:$A_2^-$] (A2-) at (9.5,3) {};
	
	\node [set,label=below:$A_2^0$] (A20) at (10.5,3) {};
	
	\node [vertex,label=left:$A_7$] (A7) at (0,3) {};
	
	\node [set,label=below:$A_7^+$] (A7+) at (0.5,3) {};
	
	\node [set,label=below:$A_7^0$] (A70) at (-0.5,3) {};
	
	\node [vertex,label=below:$A_1$] (A1) at (5,0) {};
	
	\node [set,label=below:$A_1^+$] (A1+) at (5.5,0.5) {};
	
	\node [set,label=below:$A_1^-$] (A1-) at (4.5,0.5) {};
	
	\node [set,label=below:$A_1^0$] (A10) at (5,-0.7) {};
	
	\draw[ultra thick] (A1)--(A2)--(A3)--(A4)--(A5)--(A6)--(A7)--(A1);
	
	\draw[blue] (A7+)--(A2-) (A6+)--(A1-) (A3-)--(A1+);
	\end{tikzpicture}
	\caption{Diagram for 7-bracelet. The thick line means the two sets are complete, no lines means anticompleteness, and the blue line means that the edges between the two sets are arbitrary subject to the constraint that the resulting graph is $(P_7,C_4,C_5)$-free graphs.}\label{fig:7-bracelet} 
\end{figure}
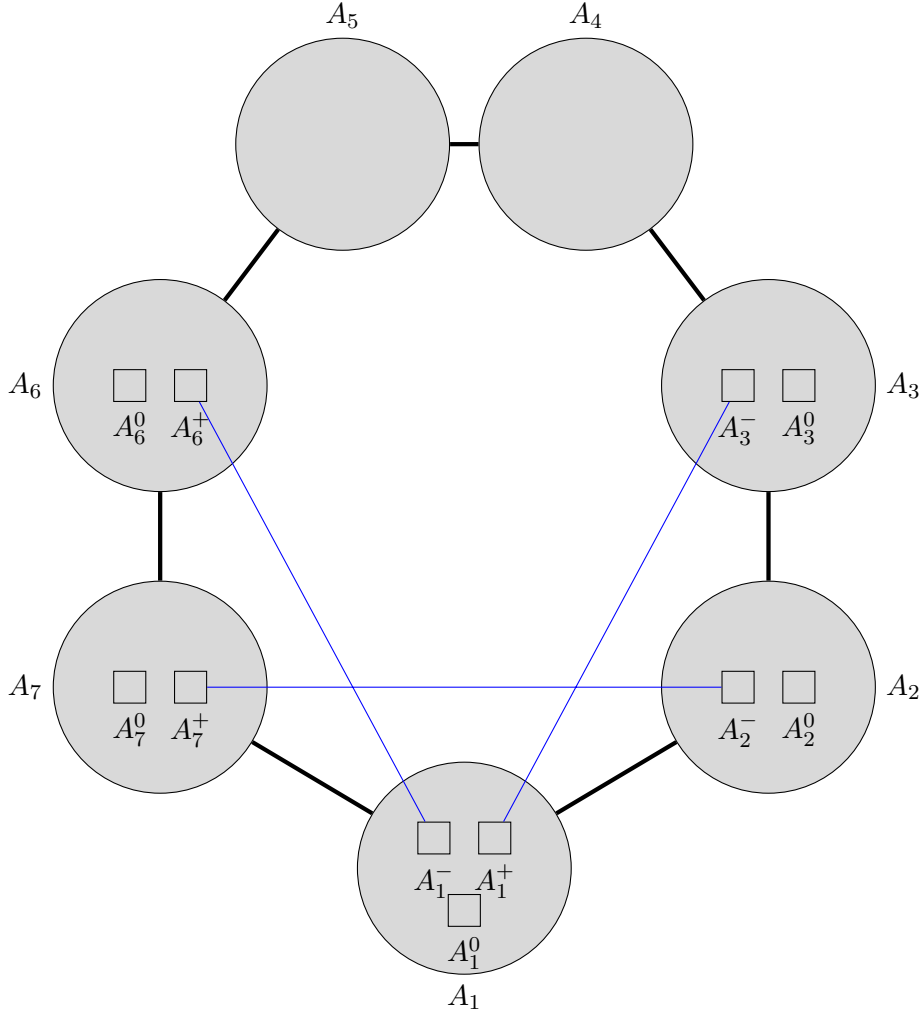

We say that $G$ is a {\em 7-bracelet} if $V(G)$ can be partition into 7 nonempty cliques $A_1,\ldots,A_7$ such that

$\bullet$ For each $1\le i\le 7$, $A_i$ is complete to $A_{i+1}\cup A_{i-1}$ and anticomplete to $A_{i+3}\cup A_{i-3}$;

$\bullet$ For $i\notin \{7,1,2\}$, $A_{i+1}$ and $A_{i-1}$ are anticomplete;

$\bullet$ $A_1$ can be partitioned into three (possibly empty) subsets $A_1^0,A_1^+,A_1^-$, $A_i$ can be partitioned into two (possibly empty) subsets $A_i^0,A_i^-$ for $i\in \{2,3\}$ and $A_i$ can be partitioned into two (possibly empty) subsets $A_i^0,A_i^+$ for $i\in \{6,7\}$ such that

\begin{enumerate}
\item $A_i^0$ is anticomplete to $A_{i+2}\cup A_{i-2}$ for each $i\in \{6,7,1,2,3\}$;
\item $A_1^+$ is anticomplete to $A_6$ and $A_1^-$ is anticomplete to $A_3$;
\item Every vertex in $A_7^+$ has a neighbor in $A_2^-$ and every vertex in $A_2^-$ has a neighbor in $A_7^+$; 
\item Every vertex in $A_1^+$ has a neighbor in $A_3^-$ and every vertex in $A_3^-$ has a neighbor in $A_1^+$, and 
every vertex in $A_1^-$ has a neighbor in $A_6^+$ and every vertex in $A_6^+$ has a neighbor in $A_1^-$. 
\end{enumerate}
See \autoref{fig:7-bracelet} for the diagram of a 7-bracelet.

\begin{theorem}[\cite{CHPS20}]\label{thm:P7C4C5}
Let $G$ be a $(P_7,C_4,C_5)$-free graph. Then $G$ contains a clique cutset or a universal vertex,
or is isomorphic to a nonempty blowup of the emerald (see \autoref{fig:emerald}) or a 7-bracelet.
\end{theorem}

It is well-known that clique cutsets and universal vertices preserve any $\chi$-binding function. Therefore,
it suffices to consider blowups of the emerald and 7-bracelets.

\begin{theorem}\label{thm:7-bracelet}
If $G$ is a 7-bracelet, then $\chi(G)\le \lceil \frac{7}{6}\omega(G) \rceil$.
\end{theorem}

\begin{theorem}\label{thm:emerald}
If $G$ is a blowup of the emerald, then $\chi(G)\le \lceil \frac{11}{9}\omega(G) \rceil$.
\end{theorem}

\begin{proof}[Proof of \autoref{thm:chibound} (assuming \autoref{thm:7-bracelet} and \autoref{thm:emerald})]

We use induction on $|G|$. We may assume that $G$ is connected, for otherwise
we apply the inductive hypothesis on each connected component.
If $G$ contains a clique cutset $K$ that disconnects $H_1$ from $H_2$,
let $G_i=G[H_i\cup K]$ for $i=1,2$. Then the inductive hypothesis implies that
$\chi(G_i)\le \lceil \frac{11}{9}\omega(G_i) \rceil$ for $i=1,2$. Note that $\chi(G)=\max\{\chi(G_1),\chi(G_2)\}$
and so $\chi(G)\le \lceil \frac{11}{9}\omega(G) \rceil$. 
If $G$ contains a strong stable set $S$,  then applying the inductive hypothesis to $G-S$ implies that
$\chi(G-S)\le \lceil \frac{11}{9}\omega(G-S) \rceil$. Since $\omega(G-S)\le \omega(G)-1$, it follows that
$\chi(G)\le \chi(G-S)+1\le \lceil \frac{11}{9}\omega(G) \rceil$.
So $G$ has no clique cutsets or strong stable sets. In particular, $G$ has no universal vertices.
It follows then from \autoref{thm:P7C4C5} that $G$ is a blowup of the emerald or a 7-bracelet.
Now the theorem follows from \autoref{thm:7-bracelet} and \autoref{thm:emerald}.
\end{proof}

We prove \autoref{thm:7-bracelet} and \autoref{thm:emerald} in the next two sections.

\section{Color 7-bracelets}

Let $G$ be a 7-bracelet with notations introduced in \autoref{fig:7-bracelet}.
By the definition of 7-bracelet, $A_7^+=\emptyset$ if and only if $A_2^-=\emptyset$, 
$A_6^+=\emptyset$ if and only if $A_1^-=\emptyset$, $A_1^+=\emptyset$ if and only if $A_3^-=\emptyset$.
If $A_6^+=A_3^-=\emptyset$, we call $G$ a {\em 7-bracelet with one uncertain pair}.

For a graph $G$, let $\beta(G)$ and $\tau(G)$ be the maximum size of a matching and minimum size of a vertex cover, respectively.
The following is a famous theorem due to K\"{o}nig.
\begin{lemma}[see for example \cite{BM08}]
For any bipartite graph $G$, $\beta(G)=\tau(G)$.
\end{lemma}

We first color 7-bracelets with one uncertain pair.

\begin{lemma}[Equal-size 7-bracelet with one uncertain pair]\label{lem:7-bracelet with one pair}
Let $G$ be a 7-bracelet $(A_1,\ldots,A_7)$ such that 

$\bullet$ $|A_i|=x$ for each $1\le i\le 7$, where $x$ is a positive integer;

$\bullet$ $A_6^+=A_3^-=\emptyset$;

$\bullet$ $\omega(G[A_7^+\cup A_2^-])\le x$.

Then $\chi(G)\le \lceil \frac{7}{6}\omega(G) \rceil$.
\end{lemma}

\begin{proof}
Since $A_6^+=A_3^-=\emptyset$, $A_i=A_i^0$ for $i\in \{6,1,3\}$. Note that $\omega(G)=2x$ since $\omega(G[A_7^+\cup A_2^-])\le x$.
We give a coloring of $G$ using $\lceil \frac{7}{6}\omega(G) \rceil =\lceil \frac{7}{3}x \rceil$ colors.
Consider the following coloring $\phi$ of $G$: color vertices in $A_i$ using colors $(i-1)x+1,(i-1)x+2,\ldots,(i-1)x+x$ for each $i$,
where colors are taken modulo $\lceil \frac{7x}{3} \rceil$.
Clearly, $\phi$ is a proper coloring of $G-F$, where $F$ is the set of edges between $A_7^+$ and $A_2^-$. 

It remains to show that it is possible to assign colors in $\phi(A_2)$ to the vertices of $A_2$ and assign colors in $\phi(A_7)$ to the vertices of $A_7$ in such a way that no edge between $A_7^+$ and $A_2^-$ is monochromatic.
If a color $c\in \phi(A_2)\cap \phi(A_7)$ is assigned to a vertex in $A_2^0$, then it can be assigned to any vertex in $A_7^+$. 
Similarly, if a color $c\in \phi(A_2)\cap \phi(A_7)$ is assigned to a vertex in $A_7^0$, then it can be assigned to any vertex in $A_2^-$. However, if a color $c\in \phi(A_2)\cap \phi(A_7)$ is not assigned to a vertex in $A_2^0\cup A_7^0$, it must be used on a non-edge between $A_7^+$ and $A_2^-$.

Let $t=|A_2^0|+|A_7^0|$. Then $G':=G[A_2^-\cup A_7^+]$ has order $2x-t$. 
Since $\omega(G')\le x$, $\alpha(\overline{G'})\le x$. 
By K\"{o}nig Theorem, $\beta(\overline{G'})=\tau(\overline{G'})$.
Since $\alpha(\overline{G'})+\tau(\overline{G'})\ge |G'|$, it follows that
\begin{align*}
\beta(\overline{G'}) & \ge |G'|-\alpha(\overline{G'}) \\
									& \ge x-t \\
									& \ge |\phi(A_2)\cap \phi(A_7)|-t.
\end{align*}

Now consider the following assignment of colors to $A_2\cup A_7$ from $\phi(A_2)\cup \phi(A_7)$. 

$\bullet$ Color the vertices in $A_2^0\cup A_7^0$ using colors in $\phi(A_2)\cap \phi(A_7)$;

$\bullet$ For each $c\in \phi(A_2)\cap \phi(A_7)$ not used on $A_2^0\cup A_7^0$ , assign $c$ to a non-edge $e_c$ between $A_7^+$ and $A_2^-$. 
Since $\beta(\overline{G'})\ge |\phi(A_2)\cap \phi(A_7)|-t$, it is possible to choose $e_c$ so that $e_c\cap e_{c'}=\emptyset$
for $c,c'\in \phi(A_2)\cap \phi(A_7)$ and $c\neq c'$.

$\bullet$ Color the remaining vertices in $A_2$ and $A_7$ in an arbitrary way.

It is routine to verify that this is a proper coloring of $G$.
\end{proof}

\begin{lemma}[Equal-size 7-bracelet]\label{lem:7-bracelet}
Let $G$ be a 7-bracelet such that 

$\bullet$ $|A_i|=x$ for each $1\le i\le 7$, where $x$ is a positive integer;

$\bullet$ $\omega(G[A_6^+\cup A_1^-]),\omega(G[A_1^+\cup A_3^-]),\omega(G[A_7^+\cup A_2^-])\le x$.

Then $\chi(G)\le \lceil \frac{7}{6}\omega(G) \rceil$.
\end{lemma}

\begin{proof}
Note that $\omega(G)=2x$ since $\omega(G[A_6^+\cup A_1^-]),\omega(G[A_1^+\cup A_3^-]),\omega(G[A_7^+\cup A_2^-])\le x$.
We give a coloring of $G$ using $\lceil \frac{7}{6}\omega(G) \rceil =\lceil \frac{7}{3}x \rceil$ colors.
Consider the following coloring $\phi$ of $G$: color vertices in $A_i$ using colors $(i-1)x+1,(i-1)x+2,\ldots,(i-1)x+x$ for each $i$,
where colors are taken modulo $\lceil \frac{7x}{3} \rceil$. 
Let $F_{61}$ ($F_{13}$) be the set of edges between $A_6$ ($A_3$) and $A_1$.
By \autoref{lem:7-bracelet with one pair}, $\phi$ is a proper coloring of $G-(F_{61}\cup F_{13})$.
It remains to show that it is possible to assign colors from $\phi(A_1),\phi(A_3),\phi(A_6)$ to vertices in $A_1,A_3,A_6$ respectively
so that the resulting coloring is a proper coloring of $G$. 

For clarity, we assume that $x=3k$ for some $k\ge 1$ (and the cases that $x=3k+1$ or $x=3k+2$ are similar).
Note that $\phi(A_1)=\{1,\ldots,3k\}$, $\phi(A_3)=\{6k+1,\ldots, 1,\ldots,2k\}$, $\phi(A_6)=\{k+1,\ldots,4k\}$. 
Then $C:=\phi(A_1)\cup \phi(A_3)\cup \phi(A_6)$ has size $5k$.
Let 
\begin{align*}
	& C_6  = C\setminus (\phi(A_1)\cup \phi(A_3)); \\
&	C_{61}  = (\phi(A_1)\cap \phi(A_6)) \setminus \phi(A_3); \\
&	C_{613}  = \phi(A_1)\cap \phi(A_6)\cap \phi(A_3); \\
&	C_{13}  = (\phi(A_1)\cap \phi(A_3)) \setminus \phi(A_6); \\
&	C_3  = C\setminus (\phi(A_1)\cup \phi(A_6)). \\
\end{align*}
Note that $C=C_6\cup C_{61}\cup C_{613}\cup C_{13}\cup C_3$ and $|C_6|=|C_{61}|=|C_{613}|=|C_{13}|=|C_3|=k$.

If $|A_1^+|\le k$, then color vertices in $A_1^+$ using colors from $C_{61}$. Since $C_{61}\cap \phi(A_3)=\emptyset$,
$A_3$ can be colored in an arbitrary way. By \autoref{lem:7-bracelet with one pair}, it is possible to assign colors
from $\phi(A_1)$ and $\phi(A_6)$ to vertices in $A_1$ and $A_6$ such that vertices in $A_1^+$ are colored with colors from
$C_{61}$ and the resulting coloring is a proper coloring of $G$. So $|A_1^+|> k$. By symmetry, $|A_1^-|> k$.

For every two vertices $x,y\in A_1^-$, $N_{A_6^+}(x)\subseteq N_{A_6^+}(y)$ or $N_{A_6^+}(y)\subseteq N_{A_6^+}(x)$, since $G$
is $C_4$-free.
Let $A_1^{--}\subseteq A_1^-$ be a set of $k$ vertices consisting of the largest $k$ vertices of $A_1^-$ (with respect to the neighborhood in $A_6^+$). Let $A_1^{++}\subseteq A_1^+$ be a set of $k$ vertices consisting of the largest $k$ vertices of $A_1^+$ (with respect to the neighborhood in $A_3^-$). Color $A_1^{--}$ using $k$ colors in $C_{13}$ and $A_1^{++}$ using $k$ colors in $C_{61}$. We now consider two cases.

\noindent {\bf Case 1.}  $|R_1^-|\le |A_6^0|$. In this case, color $R_1^-$ using colors in $C_{613}$ and these colors can also be used to color $|R_1^-|$ vertices in $A_6^0$.  Then color $A_6^+$ using the remaining colors in $\phi(A_6)$. It is easy to see that
the coloring of $A_6$ and the coloring of $A_1$ are proper.
Since in this coloring all colors used on $A_1^-$ belong to $\phi(A_1)\cap \phi(A_3)$,
it is possible to extend this coloring to a proper coloring of $G$ by \autoref{lem:7-bracelet with one pair}.

\noindent {\bf Case 2.}  $|R_1^-|> |A_6^0|$. By symmetry, $|R_1^+|> |A_3^0|$.
Since $|A_1^{--}|,|A_1^{++}|=k$ and $|A_1|=3k$,  $|R_1^-|+|R_1^+|+|A_1^0|=k$.
So
\[
	|A_6^0| < |R_1^-| \le k.
\]
This implies that $|A_6^+|>2k$ and so $|A_6^+|+|A_1^{--}|>3k$. By symmetry, $|A_3^-|>2k$ and $|A_3^-|+|A_1^{++}|>3k$.

We claim that for each $r\in R_1^-$, $r$ has a non-neighbor $f(r)$ in $A_6^+$. Moreover, we can choose $f$ to be injective.
We prove this by induction on $|R_1^-|$. 
The base case is $|R_1^-|=\{r_1\}$. If $r_1$ is complete to $A_6^+$, then $A_6^+\cup A_1^{--}\cup \{r_1\}$
is a clique of size larger than $3k$ by the definition of $A_1^{--}$. This contradicts that $\omega(G[A_6^+,A_1^-])\le 3k$. So $r_1$ has a non-neighbor $s_1$ in $A_6^+$. Now suppose that $|R_1^-|=i\ge 2$ and let $r_j$ be
the $j$th largest vertex in $R_1^-$ for $1\le j\le i$, and that we have found distinct non-neighbors $s_1,\ldots,s_{i-1}$ of $r_1,\ldots,r_{i-1}$. If $r_i$ is complete to $A_6^+\setminus \{s_1,\ldots,s_{i-1}\}$, then $(A_6^+\setminus \{s_1,\ldots,s_{i-1}\})\cup \{r_1,\ldots,r_i\}\cup A_1^{--}$ is a clique of $G[A_6^+\cup A_1^-]$ of size larger than $3k$.
This contradicts that $\omega(G[A_6^+\cup A_1^-])\le 3k$. So $r_i$ must have a non-neighbor $s_i\in A_6^+\setminus \{s_1,\ldots,s_{i-1}\}$. This proves the claim.
By symmetry, for each $r\in R_1^+$, $r$ has a non-neighbor $g(r)$ in $A_3^-$. Moreover, we can choose $g$ to be injective.

Now we can have a desired coloring as follows.

$\bullet$ Color $R_1^-,R_1^+,A_1^0$ using the $k$ colors in $C_{613}$ in any way.

$\bullet$ For each color $i$ assigned to some vertex $r\in R_1^-$, color a non-neighbor $f(r)\in A_6^+$ of $r$ with $i$, 
and color the remaining vertices in $A_6$ in any way.

$\bullet$ For each color $i$ assigned to some vertex $r\in R_1^+$, color a non-neighbor $g(r)\in A_3^-$ of $r$ with $i$, 
and color the remaining vertices in $A_3$ in any way.
\end{proof}

Now we are ready to prove \autoref{thm:7-bracelet}.
\begin{proof}[Proof of \autoref{thm:7-bracelet}]
Let $G$ be a 7-bracelet $(A_1,\ldots,A_7)$.
We prove by induction on $|G|$.
If $G$ has a strong stable set $S$, we are done by applying the inductive hypothesis to $G-S$. 
So we assume that $G$ has no strong stable set.
Let $K_{61}$, $K_{13}$ and $K_{72}$ be any maximum clique of $G[A_6^+\cup A_1^-]$, $G[A_1^+\cup A_3^-]$, and $G[A_7^+\cup A_2^-]$, respectively. Observe that the possible maximum cliques of $G$ are $A_i\cup A_{i+1}$, $K_{61}\cup A_7$, $K_{13}\cup A_2$ or $K_{72}\cup A_1$. 
We claim that
\begin{equation}\label{clm:equal size}
A_i\cup A_{i+1} \text{ is a maximum clique of } G \text{ for each } i.
\end{equation}

Since $G$ is $C_4$-free and each vertex in $A^+_6$ has a neighbor in $A_1^-$, there exits a vertex in $A_1^-$ that is complete to $A_6^+$. Let $a_1^-$ be such a vertex. Note that $a_1^-$ is universal in $G[A^+_6\cup A^-_1\cup A_7]$ and so is in every maximum clique of $G[A^+_6\cup A^-_1\cup A_7]$. The vertices $a_6^+,a_7^+,a_2^-,a_1^+,a_3^-$ can be defined similarly.
If $A_i\cup A_{i+1}$ is not maximum, we can find a strong stable set $S$ of $G$ as follows, which is a contradiction.
For $1\le i\le 7$, let $a_i\in A_i$.

$\bullet$ If $A_4\cup A_5$ is not a maximum clique, then $S=\{a_1,a_6^+,a_3^-\}$ is a desired set.

$\bullet$ If $A_5\cup A_6$ is not a maximum clique, then $S=\{a_4,a_2^-,a_7^+\}$ is a desired set. 
The case $A_4\cup A_3$ is symmetric.

$\bullet$ If $A_6\cup A_7$ is not a maximum clique, then $S=\{a_5,a_3^-,a_1^-\}$ is a desired set. 
The case $A_2\cup A_3$ is symmetric.

$\bullet$ If $A_1\cup A_7$ is not a maximum clique, then $S=\{a_4,a_2^-,a_6^+\}$ is a desired set. 
The case $A_2\cup A_1$ is symmetric.

This proves (\ref{clm:equal size}). Therefore, $|A_i|=x\ge 1$ for each $1\le i\le 7$ and 
$\omega(G[A_6^+\cup A_1^-]),\omega(G[A_1^+\cup A_3^-]),\omega(G[A_7^+\cup A_2^-])\le x$.
It now follows from \autoref{lem:7-bracelet} that $\chi(G)\le \lceil \frac{7}{6}\omega(G) \rceil$. 
\end{proof}

\section{Color Blowup of the Emerald}

Let $G$ be a blowup of the emerald. For each $i\in V(E)$ (see \autoref{fig:emerald}), 
let $L_i$ be the clique that substitutes the vertex $i$.
Define $p(G)=\min_{1\le i\le 11} |L_i|$. 
The overall strategy is to consider the blowups of a sequence of proper induced subgraphs (starting from $C_7$) of the emerald.
We prove the bounds on the blowups of the larger induced subgraph using the results on blowups of the smaller induced subgraph as the base case for induction. In the end, we are able to show that the $\lceil \frac{11}{9}\omega(G) \rceil$ bound holds 
when $p(G)\le 2$ (\autoref{lem:min size at most 2}) and $p(G)\ge 3$ (\autoref{lem:min size at least 3}), which together
imply \autoref{thm:emerald}.

We start with a famous lemma due to Lov\'asz.

\begin{lemma}[\cite{Lovasz72}]\label{lem:blowup of perfect graphs}
Any blowup of a perfect graph is still a perfect graph.
\end{lemma}

We now determine the chromatic number of equal-size blowups of $C_7$.
\begin{lemma}[Equal-size blowup of $C_7$]\label{lem:equal blowup of C7}
For any integer $t\ge 1$, $\chi(C_7[K_t]) = \lceil \frac{7}{6}\omega(C_7[K_t]) \rceil$.
\end{lemma}

\begin{proof}
Let $G=C_7[K_t]$ where $t\ge 1$. Note first that $\chi(G)\ge \lceil \frac{7t}{3} \rceil = \lceil \frac{7}{6}\omega(G) \rceil$.
Next we show that $\chi(G)\le \lceil \frac{7t}{3} \rceil$.
Let $L_i$ be the clique that substitutes the vertex $i\in V(C_7)$ for $1\le i\le 7$.
Consider the following coloring $\phi$ of $G$: color vertices in $L_i$ using colors $(i-1)t+1,(i-1)t+2,\ldots,(i-1)t+t$ for each $i$,
where colors are taken modulo $\lceil \frac{7t}{3} \rceil$. 
It is routine to verify that $\phi$ is a proper coloring of $G$.
\end{proof}

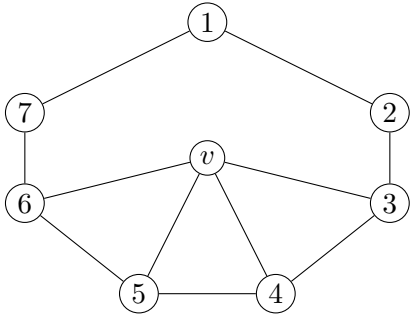
\begin{figure}[tb]
 \centering
 \begin{tikzpicture}[scale=0.6]
  \tikzstyle{vertex}=[draw, circle, fill=white!100, minimum width=4pt,inner sep=2pt]
  
  \node[vertex] (v1) at (0,3) {1};
  \node[vertex] (v2) at (4,1) {2};
  \node[vertex] (v3) at (4,-1) {3};
  \node[vertex] (v4) at (1.5,-3) {4};
  \node[vertex] (v5) at (-1.5,-3) {5};
  \node[vertex] (v6) at (-4,-1) {6};
  \node[vertex] (v7) at (-4,1) {7};
  \draw (v1)--(v2)--(v3)--(v4)--(v5)--(v6)--(v7)--(v1);
 
  \node[vertex] (v0) at (0,0) {$v$};
  \draw (v0)--(v3) (v0)--(v4) (v0)--(v5) (v0)--(v6);
 \end{tikzpicture}
 \caption{The graph $C_7+v$.}\label{fig:C7+v}
\end{figure}
 
\begin{lemma}[Blowup of $C_7+v$]\label{lem:blowup of C7+v}
If $G$ is a blowup of $C_7+v$ shown in \autoref{fig:C7+v}, then $\chi(G)\le \lceil \frac{7}{6}\omega(G) \rceil$.
\end{lemma}

\begin{proof}
Let $G$ be a blowup of $C_7+v$ shown in \autoref{fig:C7+v}. Suppose that $L_u$ is the clique that substitutes
the vertex $u\in V(C_7+v)$. If $L_i=\emptyset$ for some $i\in \{1,2,\ldots,7\}$, then $G$ is perfect and $\chi(G)=\omega(G)$.
Therefore, we may assume that $L_i\neq \emptyset$ for each $i\in \{1,2,\ldots,7\}$.
Hence, $t:=\min_{1\le i\le 7} |L_i| \ge 1$.
For each $i\in \{1,2,\ldots,7\}$, let $L'_i\subseteq L_i$ be a set of size $t$.
Denote by $H$ the subgraph of $G$ induced by $\bigcup_{i=1}^7 L'_i$. 
Note that $G-V(H)$ is perfect and $\omega(G-H)\le \omega(G)-2t$.
It follows from \autoref{lem:equal blowup of C7} that
\begin{align*}
	\chi(G) & \le \chi(G-H)+\chi(H)\\
				& \le (\omega(G)-2t) + \lceil \frac{7}{3} t \rceil\\
				& =\omega(G)+ \lceil \frac{t}{3} \rceil.
\end{align*}
Since $t\le \frac{\omega(G)}{2}$, $\chi(G)\le \lceil \frac{7}{6}\omega(G) \rceil$.
\end{proof}

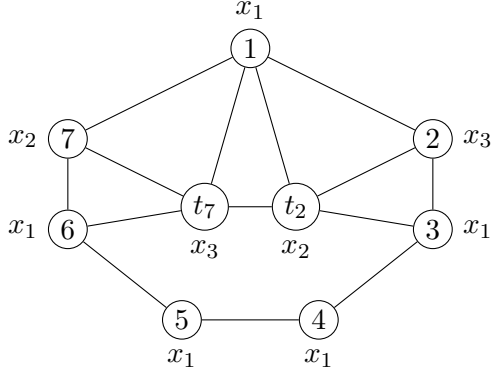
\begin{figure}[tb]
 \centering
 \begin{tikzpicture}[scale=0.6]
  \tikzstyle{vertex}=[draw, circle, fill=white!100, minimum width=4pt,inner sep=2pt]
  
  \node[vertex,label=above:$x_1$] (v1) at (0,3) {1};
  \node[vertex,label=right:$x_3$] (v2) at (4,1) {2};
  \node[vertex,label=right:$x_1$] (v3) at (4,-1) {3};
  \node[vertex,label=below:$x_1$] (v4) at (1.5,-3) {4};
  \node[vertex,label=below:$x_1$] (v5) at (-1.5,-3) {5};
  \node[vertex,label=left:$x_1$] (v6) at (-4,-1) {6};
  \node[vertex,label=left:$x_2$] (v7) at (-4,1) {7};
  \draw (v1)--(v2)--(v3)--(v4)--(v5)--(v6)--(v7)--(v1);
 
  \node[vertex,label=below:$x_3$] (s1) at (-1,-.5) {$t_7$};
  \node[vertex,label=below:$x_2$] (s2) at (1,-.5) {$t_2$};
  \draw (s1)--(v7) (s1)--(v6) (s1)--(v1) ;
  \draw (s2)--(v3) (s2)--(v2) (s2)--(v1) ;
  \draw (s1)--(s2);
 \end{tikzpicture}
 \caption{The graph $C_7+2t$. The label outside a vertex $u$ represents $|L_u|$.}
 \label{fig:C7+two 3-vertices}
\end{figure}

\begin{lemma}[Blowup of $C_7+2t$]\label{lem:blowup of C7+2v}
If $G$ is a blowup of $C_7+2t$ shown in \autoref{fig:C7+two 3-vertices}, then $\chi(G)\le \lceil \frac{7}{6}\omega(G) \rceil$.
\end{lemma}

\begin{proof}
We prove by induction on $|G|$.
Suppose that $L_u$ is the clique that substitutes the vertex $u\in V(C_7+2t)$. If $L_u=\emptyset$ for some $u\in V(C_7+2t)$, then either $G$ is perfect or $G$ is a blowup of $C_7+v$. Hence, we are done by \autoref{lem:blowup of C7+v}.
Therefore, we may assume that $L_u\neq \emptyset$ for each $u\in V(C_7+2t)$.
For every maximal clique $K$ of $C_7+2t$, there is an stable set $S_K$ such that $S_K$ meets every maximal clique of $C_7+2t$
except $K$. This implies that if $\bigcup_{u\in K} L_u$ is not a maximum clique of $G$ for some maximal clique $K$ of $C_7+2t$,
then $\omega(G-S_K)\le \omega(G)-1$. We are then done by applying the inductive hypothesis to $G-S_K$.
Therefore, $\bigcup_{u\in K} L_u$ is a maximum clique of $G$ for every maximal clique of $C_7+2t$. 
Hence, the sizes of $L_u$ must be the ones shown in \autoref{fig:C7+two 3-vertices}, where $x_1,x_2,x_3$ are positive integers
with $x_1=x_2+x_3$. Note that $\omega(G)=2(x_2+x_3)$.

For each clique $L_u$ with $|L_u|=x_1$ or $|L_u|=x_2$, let $L'_u\subseteq L_u$ be a set of size $x_2$.
Denote by $H$ the subgraph of $G$ induced by the union of those $L'_u$. Since $x_1=x_2+x_3$,
$H=C_7[K_{x_2}]$ and $G-H=C_7[K_{x_3}]$.
It follows from \autoref{lem:equal blowup of C7} that,
\begin{align*}
	\chi(G) & \le \chi(H)+\chi(G-H)\\
			    & = \lceil \frac{7}{6}(2x_2) \rceil + \lceil \frac{7}{6}(2x_3) \rceil\\
			    & \le \lceil \frac{7}{6}\omega(G) \rceil,
\end{align*}
unless $x_2,x_3\equiv 1 \pmod 3$. If $x_2=x_3=1$, then it is easy to see that $\omega(G)=4$ and $\chi(G)=5=\lceil \frac{7}{6} \omega(G) \rceil$. By symmetry, we now assume that $x_2\equiv 1 \pmod 3$ and $x_2\ge 4$. For each $u\in V(C_7+2v)$ with $|L_u|\in \{x_2,x_1\}$,
let $L'_u\subseteq L_u$ of size 3. Denote by $H$ the subgraph of $G$ induced by those $L'_u$. Then $H=C_7[K_3]$ and $\omega(G-H)\le \omega(G)-6$. By the inductive hypothesis and \autoref{lem:equal blowup of C7},
\begin{align*}
	\chi(G) & \le \chi(G-H) + \chi(H) \\
				& \le \lceil \frac{7}{6} (\omega(G)-6) \rceil + 7\\
				& \le  \lceil \frac{7}{6} \omega(G) \rceil.
\end{align*}
This completes the proof.
\end{proof}

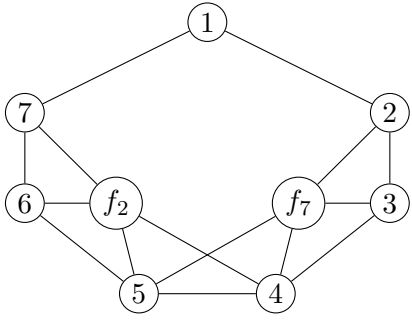
\begin{figure}[tb]
 \centering
 \begin{tikzpicture}[scale=0.6]
  \tikzstyle{vertex}=[draw, circle, fill=white!100, minimum width=4pt,inner sep=2pt]
  
  \node[vertex] (v1) at (0,3) {1};
  \node[vertex] (v2) at (4,1) {2};
  \node[vertex] (v3) at (4,-1) {3};
  \node[vertex] (v4) at (1.5,-3) {4};
  \node[vertex] (v5) at (-1.5,-3) {5};
  \node[vertex] (v6) at (-4,-1) {6};
  \node[vertex] (v7) at (-4,1) {7};
  \draw (v1)--(v2)--(v3)--(v4)--(v5)--(v6)--(v7)--(v1);
 
  \node[vertex] (s1) at (-2,-1) {$f_2$};
  \node[vertex] (s2) at (2,-1) {$f_7$};
  \draw (s1)--(v4) (s1)--(v5) (s1)--(v6) (s1)--(v7);
  \draw (s2)--(v3) (s2)--(v4) (s2)--(v5) (s2)--(v2);
 \end{tikzpicture}
 \caption{The graph $C_7+2f$.}\label{fig:C7+2f}
\end{figure}

\begin{lemma}[Blowup of $C_7+2f$]\label{lem:blowup of C_7+2f}
If $G$ is a blowup of $C_7+2f$ shown in \autoref{fig:C7+2f}, then $\chi(G)\le \lceil \frac{7}{6}\omega(G) \rceil$.
\end{lemma}

\begin{proof}
Suppose that $L_u$ is the clique that substitutes the vertex $u\in V(C_7+2f)$.
Let $t=\min \{|L_1|,|L_{f_2}|,|L_{f_7}|\}$. If $t=0$, then $G$ is either perfect or a blowup of $C_7+v$.
We are done by \autoref{lem:blowup of C7+v}.  Therefore, we assume that $t\ge 1$.
For each $u\in \{1,f_2,f_7\}$, let $L'_u\subseteq L_u$ be a set of size $t$.
Let $G'=G[L'_1\cup L'_{f_2}\cup L'_{f_7}]$. 
Note that $\omega(G-G')\le \omega(G)-t$. Moreover, $G-G'$ is either perfect or blowup of a $C_7+v$. 
By \autoref{lem:blowup of C7+v}, $\chi(G-G')\le \lceil \frac{7}{6}(\omega(G)-t) \rceil$.
It follows that
\begin{align*}
	\chi(G) & \le \chi(G-G') + \chi(G')\\ 
				& \le \lceil \frac{7}{6}(\omega(G)-t) \rceil + t \\
			   &  \le \lceil \frac{7}{6} \omega(G) \rceil.
\end{align*}
This completes the proof.
\end{proof}

\begin{lemma}[Blowup of emerald minus two vertices]\label{lem:emerald-2v}
If $G$ is a blowup of a proper induced subgraph of $H=E-8$ shown in \autoref{fig:emerald}, then $\chi(G)\le \lceil \frac{7}{6}\omega(G) \rceil$.
\end{lemma}

\begin{proof}
Note that there are five symmetric pairs $\{1,2\},\{3,6\},\{7,9\},\{10,11\},\{4,5\}$ in $H$.
Let $G$ be a blowup of $H-v$ where $v$ is a vertex in one of the five pairs.

$\bullet$ Observe that $H-1$ and $H-10$ are perfect. By \autoref{lem:blowup of perfect graphs}, any blowup of $H-1$ or $H-10$ is perfect. If $v\in \{1,2,10,11\}$, then $\chi(G)=\omega(G)$.

$\bullet$ Observe that $H-5$ is isomorphic to $C_7+v$ (see \autoref{fig:C7+v}). By \autoref{lem:blowup of C7+v}, 
$\chi(G)\le \lceil \frac{7}{6}\omega(G) \rceil$ if $v\in \{4,5\}$.

$\bullet$ Observe that $H-6$ is isomorphic to $C_7+2t$ (see \autoref{fig:C7+two 3-vertices}). By \autoref{lem:blowup of C7+2v},
$\chi(G)\le \lceil \frac{7}{6}\omega(G) \rceil$ if $v\in \{3,6\}$.

$\bullet$ Note that $H-7$ is isomorphic to $C_7+2f$ shown in \autoref{fig:C7+2f}. By \autoref{lem:blowup of C_7+2f},
 $\chi(G)\le \lceil \frac{7}{6}\omega(G) \rceil$ if $v\in \{7,9\}$.
\end{proof}

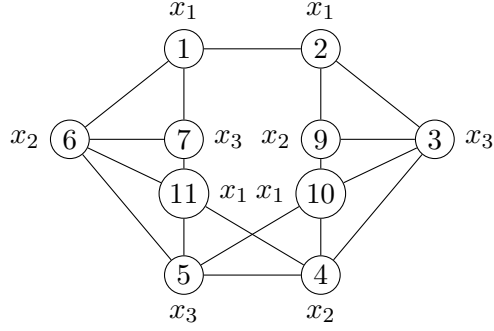
\begin{figure}[t]
	 \centering
	 \begin{tikzpicture}[scale=0.6]
	  \tikzstyle{vertex}=[draw, circle, fill=white!100, minimum width=4pt,inner sep=2pt]
	  
	  \node[vertex,label=above:$x_1$] (v1) at (-1.5,2) {1};
	  \node[vertex,label=above:$x_1$] (v2) at (1.5,2) {2};
	  \node[vertex,label=right:$x_3$] (v3) at (4,0) {3};
	  \node[vertex,label=below:$x_2$] (v4) at (1.5,-3) {4};
	  \node[vertex,label=below:$x_3$] (v5) at (-1.5,-3) {5};
	  \node[vertex,label=left:$x_2$] (v6) at (-4,0) {6};
	  \draw (v1)--(v2)--(v3)--(v4)--(v5)--(v6)--(v1);
	  
	  \node[vertex,label=right:$x_3$] (v14) at (-1.5,0) {7};
	  \draw (v14)--(v1) (v14)--(v6);
	  \node[vertex,label=left:$x_2$] (v25) at (1.5,0) {9};
	  \draw (v25)--(v2) (v25)--(v3);
	  \node[vertex,label=right:$x_1$] (s1) at (-1.5,-1.2) {11};
	  \draw (s1)--(v4) (s1)--(v5) (s1)--(v6) (s1)--(v14);
	  \node[vertex,label=left:$x_1$] (s2) at (1.5,-1.2) {10};
	  \draw (s2)--(v3) (s2)--(v4) (s2)--(v5) (s2)--(v25);
	 \end{tikzpicture}
	 \caption{The graph $E-8$. The label outside a vertex $u$ represents $|L_u|$.}
	 \label{fig:emerald-v}
	\end{figure}

\begin{lemma}[Blowup of the emerald minus one vertex]\label{lem:emerald-v}
If $G$ is a blowup of $E-8$ shown in \autoref{fig:emerald-v}, then $\chi(G)\le \lceil \frac{7}{6}\omega(G) \rceil$.
\end{lemma}

\begin{proof}
Let $H=E-8$. We prove by induction on $|G|$. Suppose that $L_u$ is the clique that substitutes the vertex $u\in V(H)$.
If $L_u=\emptyset$ for some $u\in V(E-8)$, then $G$ is a blowup of a proper induced subgraph of $H$. It follows from \autoref{lem:emerald-2v} that $\chi(G)\le \lceil \frac{7}{6}\omega(G) \rceil$. 
Therefore, we assume that $|L_u|\ge 1$ for each $u\in V(H)$.
Note that for each maximal clique $K$ of $H$, there exists an stable set $S_K$ of $H$  of size 3 such that $S$ meets every maximal clique of $H$ except $K$.

\noindent {\bf Case 1.} There exists a maximal clique $K$ of $H$ such that $\bigcup_{u\in K} L_u$ is not a maximum clique of $G$. 
Then there exsits an stable set $S$ of $G$ of size 3 such that $\omega(G-S)\le \omega(G)-1$. By the inductive hypothesis,
\begin{align*}
\chi(G) & \le \chi(G-S)+1\\
			 & \le \lceil \frac{7}{6}(\omega(G)-1) \rceil + 1\\
			 & \le \lceil \frac{7}{6}\omega(G) \rceil.
\end{align*}

\noindent {\bf Case 2.} For every maximal clique $K$ of $H$, $\bigcup_{u\in K} L_u$ is a maximum clique of $G$. 
Then the sizes of $L_u$ must be the ones shown in \autoref{fig:emerald-v}, where $x_1,x_2,x_3$ are positive integers
and $x_1=x_2+x_3$. For each clique $L_u$ with $|L_u|=x_1$ or $|L_u|=x_2$, let $L'_u\subseteq L_u$ be a set of size $x_2$.
Denote by $G'$ the subgraph of $G$ induced by the union of those $L'_u$. Since $x_1=x_2+x_3$,
$G'=C_7[K_{x_2}]$ and $G-G'=C_7[K_{x_3}]$. Note that $\omega(G)=2(x_2+x_3)$.
It follows from \autoref{lem:equal blowup of C7} that,
\begin{align*}
	\chi(G) & \le \chi(G')+\chi(G-G')\\
			    & = \lceil \frac{7}{6}(2x_2) \rceil + \lceil \frac{7}{6}(2x_3) \rceil\\
			    & \le \lceil \frac{7}{6}\omega(G) \rceil,
\end{align*}
unless $x_2,x_3\equiv 1 \pmod 3$. If $x_2=x_3=1$, then it is easy to see that $\omega(G)=4$ and $\chi(G)=5=\lceil \frac{7}{6} \omega(G) \rceil$. By symmetry, we now assume that $x_2\equiv 1 \pmod 3$ and $x_2\ge 4$. 
For each $u\in V(E-8)$ with $|L_u|\in \{x_2,x_1\}$,
let $L'_u\subseteq L_u$ of size 3. Denote by $H$ the subgraph of $G$ induced by those $L'_u$. Then $H=C_7[K_3]$ and $\omega(G-H)\le \omega(G)-6$. By the inductive hypothesis and \autoref{lem:equal blowup of C7},
\begin{align*}
	\chi(G) & \le \chi(G-H) + \chi(H) \\
				& \le \lceil \frac{7}{6} (\omega(G)-6) \rceil + 7\\
				& \le  \lceil \frac{7}{6} \omega(G) \rceil.
\end{align*}
This completes the proof.
\end{proof}

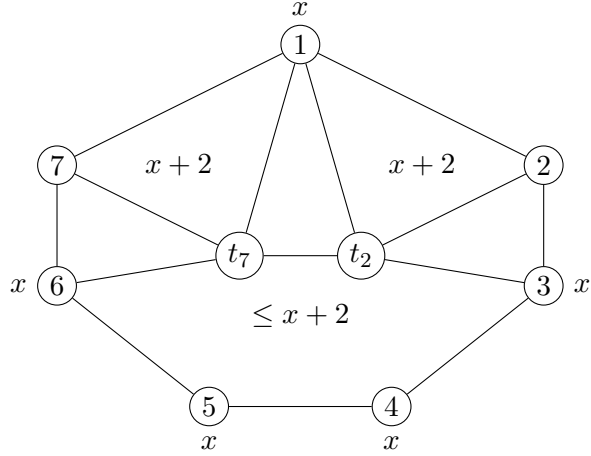
\begin{figure}[tb]
 \centering
 \begin{tikzpicture}[scale=0.8]
  \tikzstyle{vertex}=[draw, circle, fill=white!100, minimum width=4pt,inner sep=2pt]
  
  \node[vertex,label=above:$x$] (v1) at (0,3) {1};
  \node[vertex] (v2) at (4,1) {2};
  \node[vertex,label=right:$x$] (v3) at (4,-1) {3};
  \node[vertex,label=below:$x$] (v4) at (1.5,-3) {4};
  \node[vertex,label=below:$x$] (v5) at (-1.5,-3) {5};
  \node[vertex,label=left:$x$] (v6) at (-4,-1) {6};
  \node[vertex] (v7) at (-4,1) {7};
  \draw (v1)--(v2)--(v3)--(v4)--(v5)--(v6)--(v7)--(v1);
 
  \node[vertex] (s1) at (-1,-.5) {$t_7$};
  \node[vertex] (s2) at (1,-.5) {$t_2$};
  \draw (s1)--(v7) (s1)--(v6) (s1)--(v1) ;
  \draw (s2)--(v3) (s2)--(v2) (s2)--(v1) ;
  \draw (s1)--(s2);
  \node at (2,1) {$x+2$};
  \node at (-2,1) {$x+2$};
  \node at (0,-1.5) {$\le x+2$};
 \end{tikzpicture}
 \caption{The graph $G_x$, where $|L_7|+|L_{t_7}|=|L_2|+|L_{t_2}|=x+2$ and $|L_{t_7}|+|L_{t_2}|\le x+2$.}
 \label{fig:Gx}
\end{figure}

We need the following lemma to handle certain subgraphs of blowups of the emerald in the subsequent proof.

\begin{lemma}\label{lem:Gx}
Let $G_x$ be a nonempty blowup of the graph in \autoref{fig:Gx}, where $L_i$ is the clique that substitutes the vertex $i\in \{1,2,3,4,5,6,7,t_2,t_7\}$.
Suppose that $x\ge 1$, $|L_7|+|L_{t_7}|=|L_2|+|L_{t_2}|=x+2$ and $|L_{t_7}|+|L_{t_2}|\le x+2$.
Then $\chi(G_x)\le \lceil \frac{7}{6}(2x+1) \rceil$.
\end{lemma}

\begin{proof}
We prove by induction on $x$. It is routine to verify that the lemma is true for $1\le x\le 6$. 
Now suppose that $x\ge 7$. Suppose that $|L_7|,|L_{t_2}|\ge 4$. For each $u\in \{1,t_2,3,4,5,6,7\}$,
let $L'_u\subseteq L_u$ be a set of size 3 and $H$ be the union of those $L'_u$. Then $H=C_7[K_3]$
and $G-H$ is isomorphic to $G_{x-3}$. By induction and \autoref{lem:equal blowup of C7},
\begin{align*}
	\chi(G) & \le \chi(H) + \chi(G-H) \\
			   &  \le \lceil \frac{7}{6} (2(x-3)+1) \rceil + 7\\
			   & =\lceil \frac{7}{6} (2x+1) \rceil.
\end{align*}
So $|L_7|\le 3$ or $|L_{t_2}|\le 3$. If $|L_7|\le 3$, then $|L_{t_7}|\ge x-1$, $|L_{t_2}|\le 3$ and $|L_2|\ge x-1$.
Since $x\ge 7$, $|L_{t_7}|,|L_2|\ge 6$ and the same argument applies. So $|L_{t_2}|\le 3$. By symmetry, $|L_{t_7}|\le 3$.
This implies that $|L_2|,|L_7|\ge x-1\ge 6$.  For $u=1,2,\ldots,7$, let $L'_u\subseteq L_u$ be a set of size 3 and $H$ be
the union of those $L'_u$. Then $H=C_7[K_3]$.  Since $|L_{t_2}|+|L_{t_7}|\le 6\le (x-3)+2$, $G-H=G_{x-3}$. 
So we are done by applying the inductive hypothesis to $G-H$.
\end{proof}

The next lemma gives a near optimal bound on the chromatic number of blowups $G$ of the emerald with $p(G)\le 2$.
\begin{lemma}\label{lem:min size at most 2-7/6 bound}
If $G$ is a blowup of the emerald with $p(G)\le 2$, then $\chi(G)\le \lceil \frac{7}{6}\omega(G) \rceil+1$.
\end{lemma}

\begin{proof}
We prove the lemma by induction on $|G|$. 
Suppose that $L_i$ is the clique that substitutes the vertex $i\in V(E)$. 
We assume that $i\in L_i$ for $1\le i\le 11$.
Let $p = p(G)$. Since $E$ is vertex-transitive, we may assume that $|L_8|=p$.
If $p\le 1$, then $\chi(G)\le \lceil \frac{7}{6}\omega(G) \rceil + 1$ by \autoref{lem:emerald-v}.
So $p=2$. If $G=E[K_2]$, then $\chi(G)=8,\omega(G)=6$ and $\chi(G)=\lceil \frac{7}{6}\omega(G) \rceil+1$.
Now assume that $G$ is not $E[K_2]$ and the lemma holds for any blowup $G'$ of the emerald such that $|G'|<|G|$
and $p(G')\le 2$.
In the following, we assume that $G$ has no strong stable sets for otherwise we are done by the inductive hypothesis.

We say that a vertex $i\in V(E)$ is {\em minimum} if $|L_i|=p$. 
A triangle $\{i,j,k\}$ in the emerald is {\em maximum} if $|L_i|+|L_j|+|L_k|=\omega(G)$.
Let $T(E)$ be the {\em triangle graph} of $E$ whose vertices are all triangles of $E$ and two triangles
are adjacent in $T(E)$ if and only if they share an edge in $E$. Note that $T(E)=C_{11}$. 
The {\em distance} of two triangles in the emerald is their distance in $T(E)$. We proceed with a sequence of claims.

\begin{claim}\label{clm:triangles of distance 1 or 4}
If a triangle in the emerald is not maximum, then any triangle of distance 1 or 4 to that triangle is maximum.
\end{claim}

\begin{proof}[Proof of \autoref{clm:triangles of distance 1 or 4}]
By symmetry, assume that $\{1,2,8\}$ is not maximum. 
If $\{1,7,8\}$ is not maximum, then $\{4,6,9\}$ is a strong stable set.
If $\{5,6,11\}$ is not maximum, then $\{4,7,9\}$ is a strong stable set. 
In either case, it contradicts that $G$ has no strong stable set.
\end{proof} 

For a vertex $i\in V(E)$, $N_E(i)$ induces a $P_4$. 
The {\em middle triangle} of $i$ is the triangle consisting of the middle edge of the $P_4$ and $i$.
The other two triangles containing $i$ are called {\em end triangles}.

\begin{claim}\label{clm:middle triangle of a minimum vertex}
The middle triangle of a minimum vertex is maximum.
\end{claim}

\begin{proof}
Suppose that $T=\{1,2,8\}$ is not maximum. 
By \autoref{clm:triangles of distance 1 or 4}, 
\[
	\text{triangles }\{1,7,8\},\{2,8,9\},\{5,6,11\},\{3,4,10\} \text{ are maximum}.  
\]
This implies that $|L_6|=|L_3|=p$ and $\{1,6,7\},\{2,3,9\}$ are maximum, 
which in turn implies that $|L_4|=|L_5|=p$ and $\{4,5,10\},\{4,5,11\}$ are maximum.
Therefore, $|L_{10}|=|L_{11}|$, $|L_1|\ge |L_{11}|$, and $|L_2|\ge |L_{10}|$.
It follows that $|L_1|+|L_2|+|L_8|\ge |L_{10}|+|L_{11}|+p= \omega(G)$. 
This contradicts that $T$ is not maximum.
\end{proof}

\begin{claim}\label{clm:both end triangles}
If both end triangles of a minimum vertex are not maximum, then $\chi(G)\le \lceil \frac{7}{6}\omega(G) \rceil+1$.
\end{claim}

\begin{proof}
By \autoref{clm:middle triangle of a minimum vertex}, $T=\{1,2,8\}$ is maximum.
Suppose that $\{1,7,8\}$ and $\{2,8,9\}$ are not maximum.
By \autoref{clm:triangles of distance 1 or 4}, triangles $\{2,3,9\},\{1,6,7\},\{3,9,10\},\{6,7,11\},\{4,5,10\},\{4,5,11\}$ are maximum. Therefore, $|L_1|=|L_2|=|L_4|=|L_5|$. Let $x=|L_1|$, $y=|L_4|$ and $z=|L_5|$. 
Note that $\omega(G)=2x+2$, $y+z=x+2$ and $G$ is the graph shown in \autoref{fig:special emerald}, 
where $r,s$ are nonnegative integers.

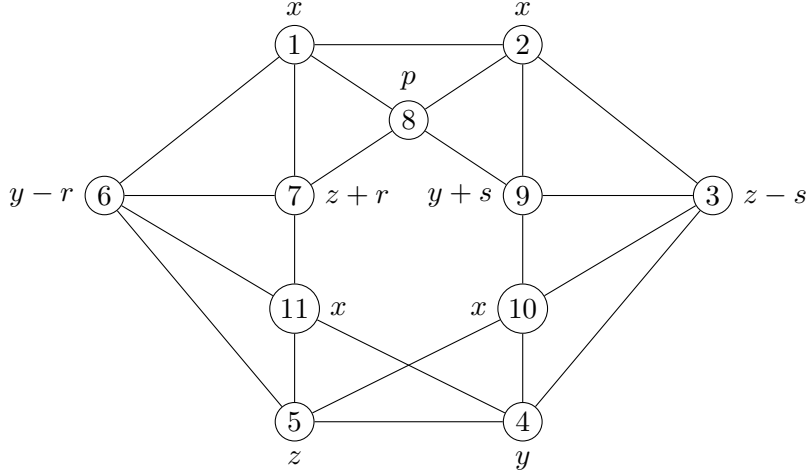
\begin{figure}[t]
	 \centering
	 \begin{tikzpicture}[scale=1]
	  \tikzstyle{vertex}=[draw, circle, fill=white!100, minimum width=4pt,inner sep=2pt]
	  
	  \node[vertex,label=above:$x$] (v1) at (-1.5,2) {1};
	  \node[vertex,label=above:$x$] (v2) at (1.5,2) {2};
	  \node[vertex,label=right:$z-s$] (v3) at (4,0) {3};
	  \node[vertex,label=below:$y$] (v4) at (1.5,-3) {4};
	  \node[vertex,label=below:$z$] (v5) at (-1.5,-3) {5};
	  \node[vertex,label=left:$y-r$] (v6) at (-4,0) {6};
	  \draw (v1)--(v2)--(v3)--(v4)--(v5)--(v6)--(v1);
	  
	  \node[vertex,label=right:$z+r$] (v14) at (-1.5,0) {7};
	  \draw (v14)--(v1) (v14)--(v6);
	  \node[vertex,label=left:$y+s$] (v25) at (1.5,0) {9};
	  \draw (v25)--(v2) (v25)--(v3);
	  \node[vertex,label=above:$p$] (v36) at (0,1) {8};
	  \draw (v36)--(v1) (v36)--(v2) (v36)--(v14) (v36)--(v25);
	  \node[vertex,label=right:$x$] (s1) at (-1.5,-1.5) {11};
	  \draw (s1)--(v4) (s1)--(v5) (s1)--(v6) (s1)--(v14);
	  \node[vertex,label=left:$x$] (s2) at (1.5,-1.5) {10};
	  \draw (s2)--(v3) (s2)--(v4) (s2)--(v5) (s2)--(v25);
	 \end{tikzpicture}
	 \caption{A special emerald with parameters $x,y,z,r,s$. The label outside the vertex $i$ represents the size of $L_i$.}
	 \label{fig:special emerald}
	\end{figure}

Since $\{2,8,9\}$ and $\{1,7,8\}$ are not maximum, $y-r,z-s\ge 3$.
Hence, $y,z\ge 3$ and $x\ge 4$. Suppose first that $r\ge 3$. 
For each vertex $i\in \{1,2,4,7,9,10,11\}$, let $L'_i\subseteq L_i$ be a set of size 3.
Note that $H:=G[L'_1\cup L'_2\cup L'_4\cup L'_7\cup L'_9\cup L'_{10}\cup L'_{11}]$ is isomorphic to $C_7[K_3]$.  
By \autoref{lem:equal blowup of C7}, $\chi(H)=7$.
Since $r\ge 3$, $|L_6|+|L_5|+|L_{11}|\le \omega(G)-3$ and so $\omega(G-H)\le \omega(G)-6$. 
By the inductive hypothesis,
\begin{align*}
\chi(G) &\le \chi(G-H) +\chi(H)\\
 		& \le \lceil \frac{7}{6} (\omega(G)-6) \rceil +1 +7\\
 		& \le \lceil \frac{7}{6} \omega(G)\rceil +1.
\end{align*}
Therefore, $r\le 2$. By symmetry, $s\le 2$. 

Suppose that $y+s<x-2$. For each vertex $i\in \{1,2,3,5,7,10,11\}$, let $L'_i\subseteq L_i$ be a set of size 3.
Note that $H:=G[L'_1\cup L'_2\cup L'_3\cup L'_5\cup L'_7\cup L'_{10}\cup L'_{11}]$ is isomorphic to $C_7[K_3]$.  
Since $y+s<x-2$, $|L_2|+|L_8|+|L_9|\le \omega(G)-3$ and so $\omega(G-H)\le \omega(G)-6$. 
By the inductive hypothesis,  $\chi(G)\le \lceil \frac{7}{6} \omega(G)\rceil +1$.
Therefore, $y+s\ge x-2$ and thus $z-s=(x+2)-(y+s)\le 4$. It follows that $z-4\le s\le 2$, i.e., $z\le 6$.
By symmetry, $y\le 6$ and so $x\le 10$. We consider two cases depending on the parity of $x$.

\noindent {\bf Case 1.} $x$ is odd. 

For $i\in \{1,2\}$, let $L'_i\subseteq L_i$ be a set of size $t=\frac{x+1}{2}$ and let $v\in L_8$. 
Let 
\[
	H=G[L_7\cup L_6\cup L_5\cup L_4\cup L_3\cup L_9\cup L'_1\cup L'_2\cup \{v\}].
\]
Note that $G-H$ is a perfect graph with $\omega(G-H)=x$.
If $\chi(H)\le x+2+t$, then
\begin{align*}
	\chi(G)	& \le (x+2+t)+x \\ 
					& =\omega(G)+t \\
					& =\frac{5}{4}\omega(G).
\end{align*}
Since $x\in \{5,7,9\}$, $\lfloor \frac{5}{4}\omega(G) \rfloor \le \lceil \frac{7}{6}\omega(G)+1 \rceil$ and we are done.

Therefore, it remains to show $\chi(H)\le x+2+t$. 
Note that $\omega(H)=x+2+t$. We first show that 
\begin{equation}\label{equ:disjoint colors}
	(z+r)+(y+s)\le (x+2)+(t-1)=\omega(H)-1.
\end{equation}
Recall that $r,s\le 2$.
If $x=9$, then $t=5$ and (\ref{equ:disjoint colors}) holds. 
If $x=7$, then $t=4$. If (\ref{equ:disjoint colors})  does not hold, then $r=s=2$. 
Since $y-r,z-s\ge 3$, it follows that $y,z\ge 5$ and so $y+z=10$. This contradiction that $y+z=x+2$.
If $x=5$, then $t=3$. If (\ref{equ:disjoint colors}) does not hold, then $r+s\ge 3$. It follows that $y+z\ge 6+(r+s)>x+2$. 
Therefore, (\ref{equ:disjoint colors}) is proved.

Now we show 
\begin{equation}\label{equ:coloring of N[v]}
\text{There is a proper } \omega(H)\text{-coloring } \phi \text{ of } L'_1\cup L'_2\cup L_7\cup L_9\cup \{v\}  \text{  with } \phi(L_7)\cap \phi(L_9) = \emptyset.
\end{equation}
We consider three cases.

\noindent If $y+s,z+r\le t$, 

$\bullet$ color the clique $L'_1\cup L'_2\cup \{v\}$ using $|L'_1\cup L'_2\cup \{v\}|\le \omega(H)$ colors;

$\bullet$ color $L_7$ using the colors in $\phi(L'_2)$ and color $L_9$ using the colors in $\phi(L'_1)$. 

\noindent If $y+s\le t, z+r\ge t$, 

$\bullet$ color the clique $L'_1\cup L_7\cup \{v\}$ using $|L'_1\cup L_7\cup \{v\}|\le \omega(H)$ colors;

$\bullet$ color $L'_2$ using the colors in $\phi(L_7)$ and color $L_9$ using the colors in $\phi(L'_1)$. 

\noindent If $y+s,z+r\ge t$,

$\bullet$ color $L_7$ and $L_9$ with disjoint sets of colors. By (\ref{equ:disjoint colors}), it is possible to choose a color
from $\{1,2,\ldots,\omega(H)\}$ to color $v$;

$\bullet$ color $L'_1$ using the colors in $\phi(L_9)$ and color $L'_2$ using the colors in $\phi(L_7)$. 

In all cases, $\phi$ is the desired coloring and (\ref{equ:coloring of N[v]}) is proved.

Let $\phi$ be a coloring of $L'_1\cup L'_2\cup L_7\cup L_9\cup \{v\}$ guaranteed by (\ref{equ:coloring of N[v]}). 
We now extend $\phi$ to a coloring of $H$ as follows.

$\bullet$ Color $L_6$ with colors from $\{1,\ldots,\omega(H)\}\setminus (\phi(L'_1)\cup \phi(L_7))$ and color $L_3$ with colors from $\{1,\ldots,\omega(H)\}\setminus (\phi(L'_2)\cup \phi(L_9))$.

$\bullet$ Color $L_5$ using $|L_5|$ colors from $\phi(L_7)$ and color $L_4$ using $|L_4|$ colors from $\phi(L_9)$. 

Since $|L_7|\ge |L_5|$, $|L_9|\ge |L_4|$ and $\phi(L_7)\cap \phi(L_9)$ are disjoint, it follows that the above coloring is
a proper $\omega(H)$-coloring of $H$.

\noindent {\bf Case 2.} $x$ is even. 

{\bf Case 2.1} $x\in \{4,6\}$.

For $i\in \{1,2\}$, let $L'_i\subseteq L_i$ be a set of size $t=\frac{x}{2}+1$.
Let 
\[
	H=G[L_7\cup L_6\cup L_5\cup L_4\cup L_3\cup L_9\cup L'_1\cup L'_2].
\] 
Note that $G-H$ is a perfect graph with $\omega(G-H)=x$.
Moreover, $H$ is a perfect graph with $\omega(H)=x+2+t$. Since $2x+2=\omega(G)$, it follows that
\begin{align*}
\chi(G) & \le (x+2+t)+x \\
		    & =\frac{5}{4}\omega(G)+\frac{1}{2}.
\end{align*}
Since $x\in \{4,6\}$, $\lfloor \frac{5}{4}\omega(G)+\frac{1}{2} \rfloor \le \lceil \frac{7}{6}\omega(G) \rceil+1$.

{\bf Case 2.2} $x\in \{8,10\}$.

For $i\in \{1,2\}$, let $L'_i\subseteq L_i$ be a set of size $t=\frac{x}{2}$. 
Let 
\[
	H=G[L_7\cup L_6\cup L_5\cup L_4\cup L_3\cup L_9\cup L'_1\cup L'_2\cup L_8].
\] 
Note that $G-H$ is a perfect graph with $\omega(G-H)=x$.
If $\chi(H)\le x+2+t$, 
\begin{align*}
	\chi(G) & \le (x+2+t)+x \\
				& =\lceil \frac{7}{6}\omega(G) \rceil +1.
\end{align*} 
Therefore, it remains to show that $\chi(H)\le x+2+t$. 
Since $r,s\le 2$ and $t=\frac{x}{2}\ge 4$, $(z+r)+(y+s)\le \omega(H)$.
Moreover, if $(z+r)+(y+s)\ge \omega(H)-1$, then $r+s\ge 3$.
Consequently, there is a proper $\omega(H)$-coloring $\phi$ of $L'_1\cup L'_2\cup L_7\cup L_8\cup L_9$ such that
$|\phi(L_7)\cap \phi(L_9)|\le 2$, and if $|\phi(L_7)\cap \phi(L_9)|\ge 1$, then $r+s\ge 3$.
Hence, it is easy to see that one can extend $\phi$ to a proper coloring of $G$ by avoiding the colors
in $\phi(L_7)\cap \phi(L_9)$ on $L_1\cup L_2\cup L_4\cup L_5$.
\end{proof}

\begin{claim}\label{clm:one end triangle}
If an end triangle of a minimum vertex is not maximum, then $\chi(G)\le \lceil \frac{7}{6}\omega(G) \rceil+1$.
\end{claim}

\begin{proof}

Suppose that $\{2,8,9\}$ is not maximum. 
By \autoref{clm:middle triangle of a minimum vertex} and \autoref{clm:both end triangles}, $\{1,2,8\}$ and $\{1,7,8\}$ are maximum. 
It follows that $|L_6|=2$ and $\{1,6,7\}$ are maximum. By \autoref{clm:middle triangle of a minimum vertex}, $\{6,7,11\}$ are maximum.
By \autoref{clm:triangles of distance 1 or 4}, $\{2,3,9\}$ and $\{4,5,10\}$ are maximum.
Since $|L_6|=\min$, $|L_4|+|L_5|+|L_{11}|\ge |L_6|+|L_5|+|L_{11}|$ 
and so $\{4,5,11\}$ is maximum by \autoref{clm:triangles of distance 1 or 4}.

Suppose first that $\{3,9,10\}$ is maximum. It follows that $|L_1|=|L_2|=|L_7|=|L_{10}|=|L_{11}|$.
Let $x=|L_1|$. Note that $\omega(G)=2x+2$, $|L_4|+|L_5|=x+2$ and $|L_3|+|L_9|=x+2$. 
Then $G-(L_6\cup L_8)$ is isomorphic to the graph $G_x$ in \autoref{fig:Gx}.
By \autoref{lem:Gx}, 
\begin{align*}
\chi(G) & \le \lceil \frac{7}{6} (2x+1) \rceil +2 \\
			  & =\lceil \frac{7}{6} (\omega(G)-1) \rceil +2 \\
			  & \le \lceil \frac{7}{6} \omega(G) \rceil +1.
\end{align*}

Now suppose that $\{3,9,10\}$ is not maximum. By \autoref{clm:both end triangles}, 9 is not a minimum vertex and so $|L_9|\ge 3$. 
By \autoref{clm:triangles of distance 1 or 4}, $\{3,4,10\}$ is maximum and so $|L_4|> |L_9|\ge 3$.
This implies that $\{5,6,11\}$ is not maximum. Now $\{2,4,7\}$ is a strong stable set, a contradiction.
\end{proof}

By \autoref{clm:middle triangle of a minimum vertex} and \autoref{clm:one end triangle}, every triangle containing a minimum vertex
is maximum. This implies that $G=E[K_2]$, a contradiction.
\end{proof}

We now show that \autoref{thm:emerald} is true for blowups $G$ of the emerald with $p(G)\le 2$.

\begin{lemma}\label{lem:min size at most 2}
If $G$ is a blowup of the emerald and $p(G)\le 2$, then $\chi(G)\le \lceil \frac{11}{9} \omega(G) \rceil$.
\end{lemma}	

\begin{proof}
By \autoref{lem:min size at most 2-7/6 bound}, $\chi(G)\le \lceil \frac{7}{6} \omega(G) \rceil +1$.
Note that $\lceil \frac{7}{6} \omega(G) \rceil +1 \le \lceil \frac{11}{9} \omega(G) \rceil$ unless $\omega(G)\in \{3,4,7,8,9,13\}$.
Next we show directly that $\chi(G)\le \lceil \frac{11}{9} \omega(G) \rceil$ when $\omega(G)\in \{3,4,7,8,9,13\}$.
Moreover, $p(G)\ge 1$ by \autoref{lem:emerald-v}.

\noindent {\bf Case 1.}  $\omega(G)=3$. It follows that $G=E$ and $\chi(G)=4=\lceil \frac{11}{9} \omega(G) \rceil$.

\noindent {\bf Case 2.} $\omega(G)=4$.  Let $I$ be the set of vertices $i$ such that $|L_i|=2$. 
Since $1\le |L_i|\le 2$  for each $i\in V(E)$, $I$ is a stable set of $E$. 
For each $i\in I$, take a vertex $u_i\in L_i$ and let $S$ be the union of those vertices.
Note that $G-S=E$. It follows that $\chi(G)\le 5=\lceil \frac{11}{9} \omega(G) \rceil$.

\noindent {\bf Case 3.} $\omega(G)\in \{7,8\}$. Note that $\lceil \frac{11}{9} \omega(G) \rceil =\omega(G) +2$.
A vertex $i\in V(E)$ is said to be a $t$-vertex if $|L_i|=t$.
Since each $L_i$ is nonempty, $G$ has no $i$-vertex for $i\ge \omega(G)-1$.
We now prove the following.

\begin{claim}\label{clm:large vertex}
If $G$ has an $i$-vertex where $i\in \{\omega(G)-4,\omega(G)-3,\omega(G)-2\}$, then $\chi(G)\le \omega(G)+2$.
\end{claim}

\begin{proof}
By symmetry, assume that $|L_8|=i\in \{\omega(G)-4,\omega(G)-3,\omega(G)-2\}$. It follows that
$|L_1|+|L_7|\le 4$ and $|L_2|+|L_9|\le 4$. Suppose first that $|L_1|=1$.
Take a vertex $u_i\in L_i$ for each $i\in \{1,2,7,9\}$. Let $S=\{u_1,u_2,u_7,u_9\}$ and $G'=G-S$. Note that $d_{G'}(v)\le \omega(G)-1$ for each $v\in L_8$. 
Therefore, $G'$ is $\omega(G)$-colorable if and only if $G'-L_8$ is.
Since $G'-L_8$ is perfect, $\chi(G'-L_8)=\omega(G'-L_8)\le \omega(G)$. 
This implies that $\chi(G)\le \chi(G')+2\le \omega(G)+2$.

Now suppose that $|L_1|,|L_2|\ge 2$. So $i=\omega(G)-4$ and $|L_1|=|L_2|=2$.
If $|L_7|=|L_9|=1$, then for every $v\in L_8$, $d_G(v)=\omega(G)+1$. So
$G$ is $\omega(G)+2$-colorable if and only if $G-L_8$ is. By \autoref{lem:emerald-v},
$\chi(G-L_8)\le \lceil \frac{7}{6} \omega(G) \rceil = \omega(G)+2$ and thus $\chi(G)\le \omega(G)+2$.
Therefore, we may assume that $|L_7|\ge 2$. Let $L'_7\subseteq L_7$ be a set of size 2.
Let $G'=G-(L_2\cup L'_7)$. For every $v\in L_8$, $d_{G'}(v)\le \omega(G)-1$. Since $G'-L_8$
is perfect, $\chi(G')\le \omega(G)$ and thus $\chi(G')\le \omega(G)$. This implies that $\chi(G)\le \omega(G)+2$.
\end{proof}

By \autoref{clm:large vertex}, we may assume that $G$ has no $i$-vertex for $i\ge \omega(G)-4$.
If $\omega(G)=7$, then $|L_i|\le 2$ for every $i\in V(E)$ by \autoref{clm:large vertex}. This implies that $\omega(G)\le 6$,
a contradiction. So $\omega(G)=8$. Since the lemma holds for any graph $G'$ with $\omega(G')=7$,
We may assume that $G$ has no strong stable set.

\begin{claim}\label{equ:3-vertex}
If $G$ has a 3-vertex, $\chi(G)\le 10$.
\end{claim}

\begin{proof}
By symmetry, assume that $|L_8|=3$.  Since $\omega(G)=8$, $|L_1|+|L_7|\le 5$ and $|L_2|+|L_9|\le 5$.
Let $v\in L_8$. If $|N_G[v]|\le 10$, we are done. So assume that $|N_G[v]|\ge 11$. 
By symmetry,  $(|L_1|+|L_7|,|L_2|+|L_9|)\in \{(5,5),(5,4),(5,3),(4,4)\}$.
Suppose first that $(|L_1|+|L_7|,|L_2|+|L_9|)\neq (5,5)$.  
Note that $|N_G[v]|\le 12$ for any $v\in L_8$.  Similar to the proof of \autoref{clm:large vertex}, $|L_1|,|L_2|\ge 2$.
If $|L_1|=|L_2|=2$, then one of $|L_7|$ and $|L_9|$ is of size at least 2, say $|L_7|$. 
Then $G-(L_2\cup L_7)$ is 8-colorable and so $\chi(G)\le 10$. 
Now by symmetry $|L_1|=3$ and $|L_2|=2$.
If $|L_7|\ge 2$, then $G-(L_2\cup L_7)$ is 8-colorable and so $\chi(G)\le 10$. 
So $|L_7|=1$ and thus $|L_1|+|L_7|+|L_8|=7<\omega(G)$. 
Since $G$ has no strong stable set, $\{1,6,7\}$ must be maximum and so $|L_6|= 4$, a contradiction.

Now assume that $|L_1|+|L_7|= 5$ and $|L_2|+|L_9|= 5$. Since $G$ has no 4-vertex, no $L_i$ is of size 1 for $i\in \{1,2,7,9\}$.
By symmetry, $(|L_7|,|L_1|,|L_2|,|L_9|)\in \{(2,3,2,3),(3,2,2,3)\}$.
Suppose first that $(|L_7|,|L_1|,|L_2|,|L_9|)=(2,3,2,3)$. Since 1 is a 3-vertex, $|L_6|=3$.
Suppose that $|L_3|=2$. Since $|L_9|=3$, it follows that $|L_{10}|=3$. 
Since $\{2,3,9\}$ is not maximum, $\{4,5,11\}$ is maximum. 
It follows that $|L_4|=3$. 
The fact that $\omega(G)=8$ implies that $|L_{11}|=3$ and $|L_5|=2$.
Let $S=L_2\cup L_5\cup L_7$. It is easy to see that $\chi(G-S)=8$ (by modulo 8 coloring) and so $\chi(G)\le 10$.
Now suppose that $|L_3|=3$. Since $|L_2|=2$ and $|L_9|=3$, $|L_{10}|=2$ and $|L_4|=3$. 
Moreover, $(|L_{11}|,|L_5|)\in \{(2,3),(3,2)\}$. In either case, it is easy to see $\chi(G)\le 10$.
So far we have proved that for every 3-vertex $i$, the neighbors of $i$ must have sizes $(3,2,2,3)$.
It then implies that $|L_4|=2$ and $|L_4|=3$, a contradiction.
\end{proof}

Now $|L_i|\le 2$ for each $i\in V(E)$. This shows that $\omega(G)\le 6$, a contradiction.
This completes the proof for $\omega(G)=8$.

\noindent {\bf Case 4.} $\omega(G)\in \{9,13\}$. Since the lemma holds for graphs with clique number $\in \{8,12\}$, we may assume that
$G$ has no strong stable set. Similar to the proof of \autoref{lem:min size at most 2-7/6 bound}, the following two claims hold.

\begin{claim}\label{clm:triangles w=13}
If a triangle in the emerald is not maximum, then any triangle of distance 1 or 4 to that triangle is maximum.
\end{claim}

\begin{claim}\label{clm:middle triangle w=13}
The middle triangle of a minimum vertex must be maximum.
\end{claim}

Now we prove the statements analogous to \autoref{clm:both end triangles} and \autoref{clm:one end triangle}.
\begin{claim}\label{clm:end triangles w=13}
If both end triangle of a minimum vertex are not maximum, then $\chi(G)\le \lceil \frac{11}{9}\omega(G) \rceil$.
\end{claim}

\begin{proof}
Suppose that $|L_8|=p$. By \autoref{clm:middle triangle w=13}, $T=\{1,2,8\}$ is maximum. 
Suppose that $\{1,7,8\}$ and $\{2,8,9\}$ are not maximum.
By \autoref{clm:triangles w=13}, 
\begin{align*}
\text{ triangles }\{2,3,9\},\{1,6,7\},\{3,9,10\},\{6,7,11\},\{4,5,10\},\{4,5,11\} \text{ are maximum.} 
\end{align*}
Then $L_1,L_2,L_4,L_5$ have the same size, say $x$. Note that $\omega(G)=2x+p$. 
Since $\omega(G)$ is odd, $p=1$ and $x=\frac{\omega(G)-1}{2}$.
Let $|L_4|=y$ and $|L_5|=z$. Then $y+z=x+1$ and $G$ is the graph shown in \autoref{fig:special emerald}, 
where $r,s$ are nonnegative integers. It is easy to show that $\chi(G)\le 11$ if $\omega(G)=9$ and $\chi(G)\le 16$ if $\omega(G)=13$.
\end{proof}

\begin{claim}
If one end triangle of a minimum vertex is not maximum, then $\chi(G)\le \lceil \frac{11}{9}\omega(G) \rceil$.
\end{claim}

\begin{proof}
Suppose that 8 is a minimum vertex and $\{2,8,9\}$ is not maximum. 
By \autoref{clm:middle triangle w=13} and \autoref{clm:end triangles w=13}, $\{1,2,8\}$ and $\{1,7,8\}$ are maximum. 
It follows that $|L_6|=p$ and $\{1,6,7\}$ are maximum. By \autoref{clm:middle triangle w=13}, $\{6,7,11\}$ are maximum.
By \autoref{clm:triangles w=13}, $\{2,3,9\}$ and $\{4,5,10\}$ are maximum.
Since $|L_6|=p$, $|L_4|+|L_5|+|L_{11}|\ge |L_6|+|L_5|+|L_{11}|$. 
By \autoref{clm:triangles w=13}, $\{4,5,11\}$ is maximum.

If $\{3,9,10\}$ is maximum, then $|L_1|=|L_2|=|L_{10}|=|L_{11}|=x$, $\omega(G)=2x+p$, and so $p =1$.
Moreover, $|L_4|+|L_5|=x+1$ and $|L_3|+|L_9|=x+1$. Then $G-(L_6\cup L_8)$ is isomorphic to a blowup of the graph in \autoref{fig:G'x}
with $x\in \{4,6\}$. It is easy to verify that $\chi(G-(L_6\cup L_8))\le \lceil \frac{11}{9}\omega(G) \rceil-1$.
So $\chi(G)\le \lceil \frac{11}{9}\omega(G) \rceil$. 
Now $\{3,9,10\}$ is not maximum. By \autoref{clm:end triangles w=13}, $|L_9|>p$. 
By \autoref{clm:triangles w=13}, $\{3,4,10\}$ is maximum and thus $|L_4|\ge |L_9|> p=|L_6|$.
So $\{5,6,11\}$ is not maximum. Now $\{2,4,7\}$ is a strong stable set, a contradiction.
\end{proof}

Now every triangle containing a minimum vertex is maximum. This implies that $G=E[K_p]$ and so $\omega(G)\le 6$.
This is a contradiction. This completes the proof of Case 4 and hence the lemma.
\end{proof}

\begin{figure}[tb]
 \centering
 \begin{tikzpicture}[scale=0.8]
  \tikzstyle{vertex}=[draw, circle, fill=white!100, minimum width=4pt,inner sep=2pt]
  
  \node[vertex,label=above:$x$] (v1) at (0,3) {1};
  \node[vertex] (v2) at (4,1) {2};
  \node[vertex,label=right:$x$] (v3) at (4,-1) {3};
  \node[vertex,label=below:$x$] (v4) at (1.5,-3) {4};
  \node[vertex,label=below:$x$] (v5) at (-1.5,-3) {5};
  \node[vertex,label=left:$x$] (v6) at (-4,-1) {6};
  \node[vertex] (v7) at (-4,1) {7};
  \draw (v1)--(v2)--(v3)--(v4)--(v5)--(v6)--(v7)--(v1);
 
  \node[vertex] (s1) at (-1,-.5) {$t_7$};
  \node[vertex] (s2) at (1,-.5) {$t_2$};
  \draw (s1)--(v7) (s1)--(v6) (s1)--(v1) ;
  \draw (s2)--(v3) (s2)--(v2) (s2)--(v1) ;
  \draw (s1)--(s2);
  \node at (2,1) {$x+1$};
  \node at (-2,1) {$x+1$};
  \node at (0,-1.5) {$\le x+1$};
 \end{tikzpicture}
 \caption{A blowup of a 9-vertex graph, where $|L_i|=x$ for $i\in \{1,3,4,5,6\}$, $|L_7|+|L_{t_7}|=|L_2|+|L_{t_2}|=x+1$ and $|L_{t_7}|+|L_{t_2}|\le x+1$.}
 \label{fig:G'x}
\end{figure}
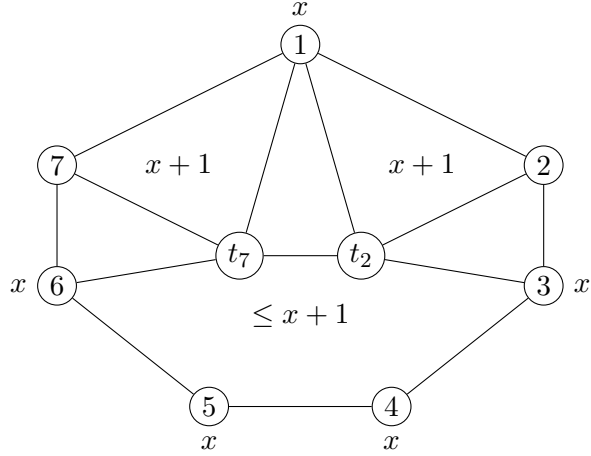

\begin{lemma}[Equal-size blowup of the emerald]\label{equal blowup of emerald}
$\chi(E[K_3])=11$.
\end{lemma}

\begin{proof}
Let $G=E[K_3]$. Since $|G|=33$ and $\alpha(G)=3$, $\chi(G)\ge 11$. Next we show that $\chi(G)\le 11$.
Observe that an 11-coloring of $G$ is equivalent to a set of 11 stable sets of size 3 such that 
each vertex of the emerald (see \autoref{fig:emerald}) is contained in exactly three stable sets. 
It is routine to verify that the following is a desired set of 11 stable sets:
\begin{align*}
&\{1,3,5\}, \{2,5,7\}, \{4,9,11\}, \{1,3,11\}, \\
& \{3,6,8\},\{6,8,10\}, \{2,4,6\}, \{4,7,9\},  \\
& \{1,5,9\}, \{2,7,10\},\{8,10,11\}.
\end{align*}
This completes the proof.
\end{proof}

Now we show \autoref{thm:emerald} for blowups $G$ of the emerald with $p(G)\ge 3$.

\begin{lemma}\label{lem:min size at least 3}
If $G$ is a blowup of the emerald with $p(G)\ge 3$, then $\chi(G) \le \lceil \frac{11}{9}\omega(G) \rceil$.
\end{lemma}

\begin{proof}
We prove by induction on $|G|$. The base case is $G=E[K_3]$ and the lemma holds.
So we assume that the lemma holds for every blowup of the emerald $G'$ with $p(G')\ge 3$ and $|G'|<|G|$.
For each $u\in V(E)$, let $L'_u\subseteq L_u$ be a set of size 3 and $H=G[\bigcup_{u\in V(E)}L'_u]$.
Then $H=E[K_3]$ and $G':=G-H$ is a blowup of the emerald. 
If $p(G')\le 2$, then $\chi(G')\le \lceil \frac{11}{9}\omega(G') \rceil$ by \autoref{lem:min size at most 2}.
If $p(G')\ge 3$, then $\chi(G')\le \lceil \frac{11}{9}\omega(G') \rceil$ by the inductive hypothesis.
Since $\omega(G')\le \omega(G)-9$, it follows from \autoref{equal blowup of emerald} that
\begin{align*}
	\chi(G) & \le \chi(G') + \chi(H) \\
			   & \le \lceil \frac{11}{9}(\omega(G)-9) \rceil + 11\\
			   & =\lceil \frac{11}{9}\omega(G) \rceil. 
\end{align*}
This completes the proof.
\end{proof}

Finally, we are ready to prove \autoref{thm:emerald}.

\begin{proof}[Proof of \autoref{thm:emerald}]
It follows directly from \autoref{lem:min size at most 2} and \autoref{lem:min size at least 3}
\end{proof}


\begin{thebibliography}{10}

\bibitem{BM08}
J.~A. Bondy and U.~S.~R. Murty.
\newblock {\em Graph Theory}.
\newblock Springer, 2008.

\bibitem{BKM06}
A.~Brandst{\"a}dt, T.~Klembt, and S.~Mahfud.
\newblock {$P_6$}- and triangle-free graphs revisited: structure and bounded
  clique-width.
\newblock {\em Discrete Mathematics and Theoretical Computer Science},
  8:173--188, 2006.

\bibitem{BDW22}
M.~Bria\'{n}ski, J.~Davies, and B.~Walczak.
\newblock Separating polynomial $\chi$-boundedness from $\chi$-boundedness.
\newblock arXiv:2201.08814v2 [math.CO], 2022.

\bibitem{CHPS20}
K.~Cameron, S.~Huang, I.~Penev, and V.~Sivaraman.
\newblock The class of $({P}_7,{C}_4,{C}_5)$-free graphs: Decomposition,
  algorithms, and $\chi$-boundedness.
\newblock {\em Journal of Graph Theory}, 93:503--552, 2020.

\bibitem{CKS07}
S.~A. Choudum, T.~Karthick, and M.~A. Shalu.
\newblock Perfectly coloring and linearly $\chi$-bound {$P_6$}-free graphs.
\newblock {\em J. Graph Theory}, 54:293--306, 2007.

\bibitem{GHM03}
S.~Gravier, C.~T. Ho{\`a}ng, and F.~Maffray.
\newblock Coloring the hypergraph of maximal cliques of a graph with no long
  path.
\newblock {\em Discrete Math.}, 272:285--290, 2003.

\bibitem{Gy73}
A.~Gy{\'a}rf{\'a}s.
\newblock On {R}amsey covering numbers.
\newblock {\em Coll. Math. Soc. J{\'a}nos Bolyai}, 10:801--816, 1973.

\bibitem{Gy87}
A.~Gy{\'a}rf{\'a}s.
\newblock Problems from the world surrounding perfect graphs.
\newblock {\em Zastosowania Matematyki}, XIX:413--441, 1987.

\bibitem{KMa18}
T.~Karthick and F.~Maffray.
\newblock Square-free graphs with no six-vertex induced path.
\newblock {\em SIAM J. Discrete Math.}, 33:874--909, 2019.

\bibitem{Lovasz72}
L.~Lov\'{a}sz.
\newblock A characterization of perfect graphs.
\newblock {\em Journal of Combinatorial Theory, Series B}, 13:95--98, 1972.

\bibitem{SSS22}
A.~Scott, P.~Seymour, and S.~Spirkl.
\newblock Polynomial bounds for chromatic number. {IV}. a near-polynomial bound
  for excluding the five-vertex path.
\newblock {\em Combinatorica}, to appear, 2022.

\bibitem{WS01}
G.~J. Woeginger and J.~Sgall.
\newblock The complexity of coloring graphs without long induced paths.
\newblock {\em Acta Cybernetica}, 15:107--117, 2001.

\end{thebibliography}

\end{document}